\documentclass{article}
\usepackage{setspace}
\usepackage{amssymb,amsthm}
\usepackage{amsmath}
\usepackage{amsfonts}
\usepackage{latexsym}
\usepackage{amssymb}
\usepackage[dvips]{graphics}
\usepackage{url}
\usepackage[plainpages]{hyperref}
\newtheorem{theorem}{Theorem}

\newtheorem{lemma}[theorem]{Lemma}
\newtheorem{corollary}[theorem]{Corollary}
\newtheorem{definition}[theorem]{Definition}

\newtheorem{example}[theorem]{Example}
\newcommand{\be}{\begin{enumerate}}
\newcommand{\ee}{\end{enumerate}}

\newcommand{\rn}{{\mathbb N}}
\newcommand{\rr}{{\mathbb R}}

\newcommand{\cz}{{\mathbb Z}}

\newcommand{\ra}{\rightarrow}
\newcommand{\nin}{\noindent}

\newcommand{\gt}{\rightarrow}
\newcommand{\gti}{\gt\infty}
\newcommand{\lni}{\lim_{n\gti}}
\newcommand{\prob}{\textrm{Prob}\,}

\newcommand{\ynpd}{Y_{n,p,d}}

\newcommand{\R}{{\mathbf R}}
\newcommand{\Z}{{\mathbf Z}}

\newcommand{\Q}{\mathbb Q}

\newcommand{\E}{{\mathbb E}}
\renewcommand{\P}{\mathbb P}

\newcommand{\LL}{\mathcal L}

\begin{document}

\title{The homotopical dimension of random 2-complexes}   
\author{Daniel C. Cohen\thanks{Partially supported by Louisiana Board of Regents grant NSF(2010)-PFUND-171.},\,  Michael Farber\thanks{Partially supported by a grant from the EPSRC.}\,  and Thomas Kappeler\thanks{Partially supported by the Swiss National Science Foundation.}}      % Enter your title between curly braces
%\author{Daniel C. Cohen, Michael Farber and Thomas Kappeler}        % Enter your name between curly braces
\date{\today}  
       % Enter your date or \today between curly braces
\maketitle
\abstract{Stochastic algebraic topology aims at studying random or partly known spaces which typically arise in applications as configuration spaces of large systems. In this paper we study the Linial--Meshulam model of random two-dimensional complexes. We prove that if the probability parameter $p$ satisfies $p\ll n^{-1-\epsilon}$, where $\epsilon>0$ is arbitrary and independent of $n$, then a random 2-complex $Y$ is homotopically one dimensional with probability tending to $1$ as $n\to \infty$. More precisely, we show that under this assumption on $p$, the complex $Y$ can be collapsed to a graph in finitely many steps. It is known that the homotopical dimension of $Y$ is equal to $2$ for $p>3 n^{-1}$.}

\section{Introduction}
%The theory of random graphs, initiated circa 1959 by Erd\"os and R\'enyi \cite{ER}, is nowadays a well-developed and fast growing branch of discrete mathematics.
Since its inception in 1959 by Erd\"os and R\'enyi \cite{ER}, the theory of random graphs has developed into 
a rapidly growing and widely applicable branch of discrete mathematics, bringing together ideas from graph theory, combinatorics, and probability theory. In one model, a random graph is a subgraph $\Gamma$ of a complete graph on $n$ vertices such that every edge of the complete graph 
is included in $\Gamma$ with probability $p$, independently of the other edges.  One is interested in probabilistic features of $\Gamma$ and their dependence on $p$ when $n$ is large. Here $0<p<1$ is a probability parameter which in general may depend on $n$. The theory of random graphs 
\cite{AS, B, JLR} offers many spectacular results and predictions, which  play an essential role in various engineering and computer science applications. Random graphs also serve within mathematics as accessible models for other, more complex random 
structures.  

Higher dimensional analogs of the aforementioned 
Erd\H{o}s--R\'enyi model were recently suggested and studied by Linial--Meshulam in~\cite{LM}, and 
Meshulam--Wallach in~\cite{MW}.
In these models, one generates a random $d$-dimensional simplicial complex $Y$ by considering the full $d$-dimensional skeleton of the simplex 
$\Delta_n$ on vertices $\{1, \dots, n\}$ and retaining $d$-dimen\-sional faces independently with probability $p$.  Note that in this construction $Y$ contains the $(d-1)$-dimensional 
skeleton of $\Delta_n$. 
The work of Linial--Meshulam and
Meshulam--Wallach provides threshold functions for the vanishing of the
$(d-1)$-st homology groups of random complexes with coefficients in a~finite
abelian group. Threshold functions for the vanishing of the $d$-th homology groups were subsequently studied by Kozlov \cite{Ko}.

In this paper, we focus on 2-dimensional random complexes. The corresponding probability space $G(\Delta_n^{(2)}, p)$ of the Linial--Meshulam model is defined as follows. Let $\Delta_n$ denote the $(n-1)$-dimensional simplex 
with vertices $\{1, 2, \dots, n\}$. Then $G(\Delta_n^{(2)}, p)$ denotes the set of all 2-dimensional subcomplexes 
$$\Delta_n^{(1)}\subset Y\subset \Delta_n^{(2)},$$ containing the one-dimensional skeleton $\Delta_n^{(1)}$. The probability function 
$\P: G(\Delta_n^{(2)},p)\to \R$ is given by the formula
$$\P(Y) = p^{f(Y)}(1-p)^{{n\choose 3}-f(Y)}, \quad Y\in G(\Delta_n^{(2)}, p),$$
where $f(Y)$ denotes the number of faces in $Y$.  In other words, each of the 2-dimensional simplexes of $\Delta_n^{(2)}$ is included in a random 2-complex $Y$ with probability $p$, independently of the other 2-simplexes. As in the case of random graphs, $0<p<1$ is a probability parameter which may depend on $n$.  
When $n$ grows, the model $G(\Delta_n^{(2)}, p)$ includes all finite $2$-dimensional complexes containing the full 1-skeleton $\Delta_n^{(1)}$; however, the likelihood of various topological phenomena is dependent on the value of $p$. The theory of deterministic 2-complexes itself is a rich and active field of current research with many challenging open questions, see \cite{Hog}.

The fundamental group of a random 2-complex $Y\in G(\Delta_n^{(2)}, p)$ was investigated by Babson, Hoffman, and Kahle \cite{BHK}.  They showed that for 
$p\gg n^{-1/2}\cdot (3\log n)^{1/2}$, the group $\pi_1(Y)$ vanishes asymptotically almost surely (i.e., the probability that $\pi_1(Y)$ is trivial tends to $1$ as $n\to\infty$). 
For $p\ll n^{-1/2-\epsilon}$, these authors use notions of negative curvature due to Gromov to study the nontriviality and hyperbolicity of $\pi_1(Y)$.
%The paper \cite{BHK} also states that $\pi_1(Y)$ is hyperbolic and nontrivial for $p\ll n^{-1/2-\epsilon}$. 

In this paper, we show that for $p\ll n^{-1-\epsilon}$ a random 2-complex $Y$ is homotopically 1-dimensional, a.a.s.\footnote{We use the abbreviation a.a.s.~for the phrase \lq\lq asymptotically almost surely\rq\rq.}
More precisely, we show that $Y$ can be collapsed to a graph in finitely many steps. This implies that $Y$ has a free fundamental group and vanishing 2-dimensional homology. Note that the vanishing of 2-dimensional homology in this range of $p$ also follows from a result of Kozlov \cite{Ko}. In \cite{CFK}, it is shown that for $p>3/n$, the homology group $H_2(Y;\Z)$ is nontrivial with probability tending to $1$; see also \cite{Ko}. Thus, 
for $p>3/n$, the random 2-complex $Y$ is homotopically two-dimensional a.a.s. 

Our main result  is as follows:

\begin{theorem}\label{main} 
%(A) If for some $k\ge 1$ the probability parameter $p$ satisfies\footnote{Recall that the symbol $a_n\ll b_n$ means that $a_n>0$ and $a_n/b_n \to 0$ as $n \to \infty$.}
%\begin{eqnarray}
%\label{ass} p\ll n^{-1-\frac{2}{k+1}}.
%\end{eqnarray} Then a random 2-complex 
%$Y\in G(\Delta_n^{(2)},p)$ is collapsible to a graph in at most $k$ steps, asymptotically almost surely (a.a.s). 
%(B) If for some $k\ge 1$ the probability parameter $p$ satisfies 
%\begin{eqnarray}
%p\gg n^{-1-\frac{1}{3\cdot 2^{k-1}-1}},
%\end{eqnarray}
%then $Y$ is not collapsible to a graph in $\le k$ steps, a.a.s.
%\begin{enumerate}
  (a) If for some $k\ge 1$ the probability parameter $p$ satisfies\footnote{Recall that the symbol $a_n\ll b_n$ means that $a_n>0$ and $a_n/b_n \to 0$ as $n \to \infty$.}
\[
\label{ass} p\ll n^{-1-\frac{2}{k+1}},
\] then a random 2-complex 
$Y\in G(\Delta_n^{(2)},p)$ is collapsible to a graph in at most $k$ steps, asymptotically almost surely (a.a.s). 
(b) If for some $k\ge 1$ the probability parameter $p$ satisfies 
\[
p\gg n^{-1-\frac{1}{3\cdot 2^{k-1}-1}},
\]
then $Y$ is not collapsible to a graph in $k$ or fewer steps, a.a.s.
\end{theorem}

Loosely speaking, Theorem \ref{main} combines with previously known results to suggest that {\it a random 2-complex with vanishing 2-dimensional homology is homotopically one-dimensional. }

Theorem \ref{main} implies:

\begin{corollary}
If for some $k\ge 1$ the probability parameter $p$ satisfies  $$p\ll n^{-1-\frac{2}{k+1}}$$
then the fundamental group $\pi_1(Y)$ of a random 2-complex 
$Y\in G(\Delta_n^{(2)},p)$ is free and $H_2(Y;\Z)=0$, a.a.s.
\end{corollary}

The proof of Theorem \ref{main} is given at the very end of the paper. A key role is played by Theorem \ref{thm1}, 
which states that there exists a finite list of forbidden 2-complexes $\LL_{k,r}$ with $k\ge 0$ and $r\ge 2$, 
such that an arbitrary 2-complex of degree at most $r$ (see below) is collapsible to a graph in $k$ steps if and only if it does not contain any of the 2-complexes from $\LL_{k,r}$. This allows us to reduce the collapsiblity problem to the containment problem for random complexes which was studied in \cite{CFK}.

\subsection*{Acknowledgments}
This research was implemented during visits of M. Farber to FIM ETH Z\"urich and Louisiana State University, and a visit of D. Cohen to the University of Z\"urich.  Portions of this work were carried out during the Spring of 2010, when the first two authors participated in the Mathematisches Forschungsinstitut Oberwolfach Research in Pairs program. 
  We thank the FIM ETH, LSU, the University of Z\"urich, and the MFO for their support and hospitality, and for providing productive mathematical environments.

\section{Collapsibility of a 2-complex to a graph}  

\subsection{Basic definitions}

Let $Y$ be a finite 2-dimensional simplicial complex. An edge of $Y$ is called {\it free} if it is included in exactly one 2-simplex.

The {\it boundary} $\partial Y$ is defined as the union of free edges.  We say that a 2-complex $Y$ is {\it closed} if $\partial Y=\emptyset$. 

A $2$-complex 
$Y$ is called {\it pure} if every maximal simplex is 2-dimensional. By the {\it pure part} of a 2-complex we mean the maximal pure subcomplex, i.e. the union of all 2-simplexes.

Let $Y$ be a simplicial 2-complex and let $\sigma$ and $\tau$ be two 2-simplexes of $Y$. 
We say that $\sigma$ and $\tau$ are adjacent if they intersect in an edge.
The {\it distance} between $\sigma$ and $\tau$, $d_Y(\sigma, \tau)$, is the minimal integer $k$ such that there exists a sequence of 2-simplexes $\sigma=\sigma_0, \sigma_1, \dots, \sigma_k=\tau$ with the property that $\sigma_i$ is adjacent to $\sigma_{i+1}$ for every $0\le i<k$. (If no such sequence exists then $d_Y(\sigma, \tau)=\infty$.) The {\it diameter }
${\rm {diam}}(Y)$ is defined as the maximal value of $d_Y(\sigma, \tau)$ taken over pairs of 2-simplexes of $Y$. 

A simplicial 2-complex is {\it strongly connected} if it has a finite diameter. 

A simplicial 2-complex  has {\it degree $\le r$} if every edge is incident to at most $r$ 2-simplexes.

A {\it pseudo-surface} is a finite, pure, strongly connected 2-dimensional simplicial complex of degree at most $2$ (i.e., every edge is included in at most two 2-simplexes). 

More generally, for an integer $r>0$, an {\it $r$-pseudo-surface} is a finite, pure, strongly connected 2-dimensional simplicial complex of degree at most $r$. %with the property that every edge is included in at most two r-simplexes. 

\subsection{Simplicial collapse} Let $Y$ be a 2-complex. A 2-simplex of $Y$ is called {\it free} if at least one of its edges is free. 
Let $\sigma_1, \dots, \sigma_k$ be all free 2-simplexes in $Y$, and let 
$e_1, \dots, e_k$ be free edges with $e_i\subset \sigma_i$. We say that the complex  
$$Y'={Y-\cup_{i=1}^k {\rm {int}}(\sigma_i) - \cup_{i=1}^k {\rm {int}}(e_i)}$$ 
is obtained from $Y$ by collapsing all free 2-simplexes. Clearly $Y'\subset Y$ is a deformation retract. The operation $Y\searrow Y'$ is called a {\it simplicial collapse}. 
Note that $Y'$ is not uniquely determined if one of the free simplexes of $Y$ has two free edges; however the pure part of $Y'$ (i.e. the union of 2-simplexes of $Y'$) is uniquely determined. 

This process can be iterated $Y'\searrow Y''$, $Y''\searrow Y'''$, etc. We denote $Y=Y^{(0)}$, $Y'=Y^{(1)}$, $Y''=Y^{(2)}$ etc. 
The sequence of subcomplexes $Y^{(0)}\supset Y^{(1)}\supset Y^{(2)}\supset \dots$ is decreasing and there are two possibilities: either (a)  for some $k$, the complex $Y^{(k)}$ is one-dimensional (a graph), or (b) for some $k$, the complex $Y^{(k)}$ is 2-dimensional and closed, i.e., $\partial Y^{(k)}=\emptyset$. 

\begin{definition} We say that $Y$ is collapsible to a graph in at most $k$ steps if $Y^{(k)}$ is a graph. 
We say that $Y$ is collapsible to a graph in $k$ steps if $Y^{(k)}$ is a graph and $\dim Y^{(k-1)}=2$. 
\end{definition}

%In case (a) we say that {\it $Y$ is collapsible to a graph in $k$ steps}. 

Observe that {\it if $Y$ is collapsible to a graph in at most $k$ steps then any simplicial subcomplex $S\subset Y$ is also collapsible to a graph in at most $k$ steps}. 
At each step one removes the free triangles in $Y^{(i)}$ which belong to $S$. 

Let $Y$ be a 2-complex, and consider the sequence of collapses
\[
Y^{(0)}\searrow Y^{(1)}\searrow Y^{(2)}\searrow\dots \searrow Y^{(k)} \searrow \dots.
\] 
For a 2-simplex $\sigma\in Y$ define $$ D_Y(\sigma) = \sup \{i; \, \sigma\subset Y^{(i)}\}\, \, \in \{0, 1, \dots, \infty\}.$$ 
A 2-simplex $\sigma$ is free if and only if $D_Y(\sigma)=0$. 

A 2-complex $Y$ is collapsible to a graph in at most $k+1$ steps if and only if $D_Y(\sigma)\le  k$ for any $2$-simplex $\sigma$. 
If after performing several collapses $Y^{(0)}\searrow Y^{(1)}\searrow Y^{(2)}\searrow\dots $ 
we obtain a subcomplex $Y^{(r)}\subset Y$ with empty boundary $\partial Y^{(r)}=\emptyset$, 
then $Y^{(r)}=Y^{(r+1)}=Y^{(r+2)}=\dots$ and $D_Y(\sigma)=\infty$ for any simplex $\sigma$ in $Y^{(r)}$.

\begin{lemma} Let $\sigma$ be a 2-simplex with $D_Y(\sigma)=k$ where $0<k<\infty$. Then one of the edges $e$ of $\sigma$ has the following property: 
for any 2-simplex $\sigma'$ of $Y$ which is incident to $e$ and distinct from $\sigma$ one has $D_Y(\sigma')<k$ and there exists a 2-simplex $\sigma'$ incident to $e$ and distinct from 
$\sigma$ such that $D_Y(\sigma')=k-1$. 
\end{lemma}
\begin{proof}
Since $D_Y(\sigma)=k$, we know that after $k$ collapses an edge $e$ of $\sigma$ becomes free. All other simplexes $\sigma'$ of $Y$ incident to $e$ must have been eliminated in previous steps, i.e., they satisfy $D_Y(\sigma')<k$. At least one of these simplexes $\sigma'$ must have been eliminated in step $k-1$ since otherwise $\sigma$ would have become free earlier. 
\end{proof}

\begin{lemma}\label{ge}
If $Z\subset Y$ is a subcomplex %, then for any 2-simplex $\sigma\subset Z$ one has 
and $\sigma\subset Z$ is a $2$-simplex, then 
$$D_Z(\sigma) \le D_Y(\sigma).$$
\end{lemma}
\begin{proof}
If a 2-simplex belongs to $Z$ and is not free in $Z$ then it is not free in $Y$. This implies that $Z'\subset Y'$ and therefore $Z^{(i)}\subset Y^{(i)}$ for any $i\ge 1$. 
Thus, the maximal $i$ such that $\sigma$ is contained in $Z^{(i)}$ is less than or equal to the maximal $i$ such that $\sigma$ is contained in $Y$, which implies the statement of the Lemma.
\end{proof}

\subsection{$\sigma$-accessible boundary}

\begin{definition}
Let $Y$ be a 2-complex and let $\sigma, \tau$ be two 2-simplexes of $Y$ with $D_Y(\tau)=0$ and $D_Y(\sigma)=k\ge 1$.
 A {\it collapsing path} from $\tau$ to $\sigma$ is a sequence of 2-simplexes $\tau=\sigma_0, \sigma_1, \dots, \sigma_{k-1}, \sigma_k=\sigma$ 
such that $D_Y(\sigma_i)=i$ and each pair $\sigma_i$ and $\sigma_{i+1}$ has a common edge, where $i=0, \dots, k-1$.
\end{definition}

In a collapsing path, the initial simplex $\sigma_0=\tau$ is a free simplex, and hence at least one of its edges belongs to the boundary $\partial Y$. 

\begin{definition}\label{defaccess} Given a 2-simplex $\sigma$, we denote by $A_Y(\sigma)\subset \partial Y$ the union of the edges in $\sigma_0\cap \partial Y$ 
which can appear in a collapsing path  $\sigma_0, \sigma_1, \dots, \sigma_k$  ending at $\sigma$. We call $A_Y(\sigma)$ the $\sigma$-accessible part of the boundary.
\end{definition}
In Definition \ref{defaccess}, clearly $k=D_Y(\sigma)$.
Note that $A_Y(\sigma) \not=\emptyset$ if and only if $D_Y(\sigma)<\infty$. 

\begin{definition} Let $\sigma$ be a 2-simplex of $Y$ with $D_Y(\sigma) \ge 1$. For an edge $e$ of $\sigma$ define 
$$A_Y(\sigma, e)\subset A_Y(\sigma)$$ as the set of all edges $e'$ of the boundary 
$\partial Y$ with the property that there exists a collapsing path $\sigma_0, \sigma_1, \dots, \sigma_k=\sigma$ such that $e'$ is an edge of $\sigma_0$ and $e=\sigma_{k-1}\cap \sigma_k$.
\end{definition}

If $e_1, e_2, e_3$ are the edges of $\sigma$ then $A_Y(\sigma)= \cup_{i=1}^3 A_Y(\sigma, e_i)$ and the sets $A_Y(\sigma,e_i)$ need not be mutually disjoint.

\begin{lemma} \label{dadj} Let $\sigma$ and $\sigma'$ be adjacent 2-simplexes of $Z$ with $$D_Z(\sigma)= D_Z(\sigma')+1.$$ 
Assume that any collapsing path in $Z$ ending at $\sigma$ passes through the edge
$e=\sigma\cap \sigma'$. 
If $Z$ is embedded as a subcomplex $Z\subset Y$ and $$D_Z(\sigma')<D_Y(\sigma'),$$ then 
$$D_Z(\sigma)< D_Y(\sigma).$$
\end{lemma}
%\begin{proof} Since $D_Y(\sigma') > D_Z(\sigma')=D_Z(\sigma)-1$,  the simplex $\sigma'$ survives after performing $D_Z(\sigma)-1$ collapses of $Y$. Furthermore, since 
%$D_Y(\sigma)\ge D_Z(\sigma)$ (by Lemma \ref{ge}), the simplex $\sigma$ survives as well. Thus, the edge $e$ does not become free after $D_Z(\sigma)-1$ collapses of $Y$. The other two edges 
%of $\sigma$ may become free in $Z$ after at least $D_Z(\sigma)$ collapses. Hence after $D_Z(\sigma)-1$ collapses of $Y$ none of the edges of $\sigma$ becomes free. Hence $D_Y(\sigma)>D_Z(\sigma)$. 
%\end{proof}

\begin{proof}  
Let $k=D_Z(\sigma') =  D_Z(\sigma) - 1$. We must show that $D_Y(\sigma) \ge k+2.$ First we claim that the edge $e$ may become free only after at least $k+2$ collapses in $Y$. Assume it is free in $Y$ after $k+1$ collapses. By assumption, $D_Y(\sigma') \ge k+1.$ Hence the edge $e$ can only be free after $k+1$ collapses in $Y$ if $\sigma$ has been removed already before, i.e., $D_Y(\sigma) \le k.$ On the other hand, by Lemma \ref{ge}, 
$D_Y(\sigma) \ge D_Z(\sigma) = k + 1$ which leads to a contradiction. 

By assumption, the two edges of $\sigma$ different from $e$ are not free in $Z^{(k+1)}$ and hence they are not free in $Y^{(k+1)}$. Thus $D_Y(\sigma) \ge k + 2$ as claimed.
\end{proof}

%Introduce the subset $A_{\sigma, e}\subset A_\sigma$ (edges of $\partial Y$ accessible through the edge $e$.

Note that the assumption of Lemma \ref{dadj} that any collapsing path in $Z$ ending at $\sigma$ passes through the edge $e$ is equivalent to $A_Z(\sigma, e')=\emptyset$ for the two remaining edges $e'\not=e$ of $\sigma$.

\begin{lemma}\label{lm8} Let $Z\subset Y$ be a subcomplex. If $D_Z(\sigma) = D_Y(\sigma)$ for a 2-simplex $\sigma$ of $Z$ then there is an edge $e$ of $\sigma$ such that 
$$\emptyset \not= A_Z({\sigma, e}) \subset A_Y({\sigma, e})\subset \partial Y.$$
\end{lemma}

\begin{proof}
Without loss of generality, we may assume that $Y$ is obtained from $Z$ by attaching a single $2$-simplex.

The proof is by induction on $k=D_Y(\sigma)=D_Z(\sigma)$.

In the case $k=0$, there is an edge $e$ of $\sigma$ that is free in both $Z$ and $Y$.  In particular, $e\subset\partial Y$.

We include the case $k=1$.  Recall that $Z'=Z^{(1)}$ denotes the result of the first collapse of $Z$, $Z\searrow Z'$.  
Since $D_Z(\sigma)=D_Y(\sigma)=1$, there is an edge $e$ of $\sigma$ that is free in $Y'$ and hence in  $Z'$.  
Then every collapsing path $\tau,\sigma$ in $Z$ with $e=\tau\cap\sigma$ is also a collapsing path in $Y$. Hence 
$A_Z({\sigma, e}) \subset A_Y({\sigma, e})$.

For the general case, assume that $D_Y(\sigma)=D_Z(\sigma)=k$.  After $k$ collapses
\[
Z\searrow Z^{(1)}\searrow \dots \searrow Z^{(k)},\quad
Y\searrow Y^{(1)}\searrow \dots \searrow Y^{(k)},
\]
the $2$-simplex $\sigma$ is exposed in both $Z^{(k)}$ and $Y^{(k)}$.  Thus, $\sigma$ has a free edge $e$
 in $Y^{(k)}$ (and hence in $Z^{(k)}$ as well).  Writing $Z'=Z^{(1)}$ and $Y'=Y^{(1)}$, by induction, we have
 $\emptyset\not=A_{Z'}({\sigma,e}) \subset A_{Y'}({\sigma,e})$ so that any collapsing path $\sigma_1,\dots,\sigma_k$ from  $\sigma_1=\sigma' \subset A_{Z'}({\sigma,e})$ to 
$\sigma_k=\sigma$ in $Z'$ is also a collapsing path in $Y'$.  
 Note in particular that every edge of $\sigma'$ that is free in $Z'$ is also free in $Y'$.  Consequently, for every free triangle $\tau$ in $Z$ which meets $\sigma'$ in 
an edge free in $Z'$, the collapsing path
 $\tau=\sigma_0,\sigma_1,\dots,\sigma_k$ in $Z$ is a collapsing path in $Y$.  The result follows.
\end{proof}

\begin{corollary}\label{corgood}
Let $Z\subset Y$ be 2-complexes such that for a 2-simplex $\sigma$ of $Z$ none of the edges $e\in A_Z(\sigma)\subset \partial Z$ is free in $Y$. Then
$$ D_Z(\sigma)+1\le D_Y(\sigma).$$
\end{corollary}
\begin{proof} For a contradiction, assume that $D_Y(\sigma) \le D_Z(\sigma)$.  Then $D_Y(\sigma) =D_Z(\sigma)$ by Lemma \ref{ge}. 
We may now apply Lemma \ref{lm8} which claims that there is an edge $e$ of 
$\sigma$ for which $\emptyset\not=A_Z(\sigma, e)\subset A_Y(\sigma, e)\subset \partial Y$. %But 
This contradicts our assumption that no edge in $A_Z(\sigma)$ lies on the boundary 
$\partial Y$. 
\end{proof}

\subsection{The list of forbidden $r$-pseudo-surfaces $\LL_{k,r}$} 

For a pair of integers $k=0, 1, \dots,$  and $r=2, 3, \dots$ we denote by $\LL_{k,r}$ the set of all isomorphism types of 
$r$-pseudo-surfaces $S$ with the following properties:
\begin{enumerate}
  \item[(a)] Each $S\in \LL_{k,r}$ has a specified 2-simplex $\sigma_\ast$ (called {\it the center}). 
  \item[(b)] If $\partial S\not=\emptyset$ then $D_S(\sigma_\ast) =k$. 
  %\item[] distance from the center $\sigma_0$ to any free triangle $\sigma\subset S$ equals $k$, i.e. $d(\sigma_0, \sigma)=k$. 
  \item[(c)] $d_S(\sigma_\ast, \sigma)\le k$ for any 2-simplex $\sigma$. 
  \end{enumerate}

\begin{figure}[h]
\begin{center}
\resizebox{11cm}{4cm}{\includegraphics[20,412][548,595]{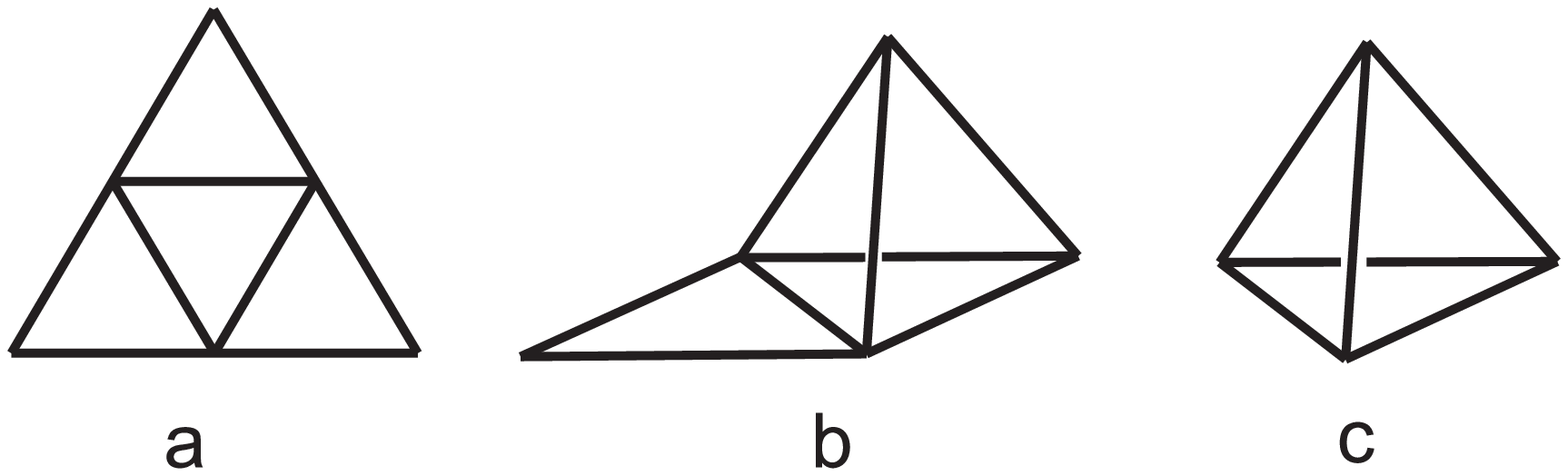}}
\end{center}
\caption{Surfaces $\LL_{1,2}$.}\label{lone}
\end{figure}

\noindent Note that $\LL_{0,r}=\{S\}$ consists of a single complex $S=\sigma_\ast$ (the triangle). 

The set $\LL_{1,2}$ consists of the three surfaces shown in Figure \ref{lone}. Each of the surfaces 
a, b, c is a union of 4 triangles. The surface c is a tetrahedron, b is a tetrahedron with one face open, and a is a fully flattened tetrahedron. 

It is clear that $\LL_{k,r}$ is finite and $\LL_{k,r} \subset \LL_{k, r+1}$. 

\begin{example}\label{ex1}{\rm Consider the following important family of surfaces $S_k\in \LL_{k, 2}$ where $k=0,1,2,\dots$. The first surface $S_0$ is defined as a single triangle $S_0=\sigma_\ast$. The next surface $S_1$ is the
shown in Figure \ref{lone} a. Surfaces $S_2$ and $S_3$ are shown in Figure \ref{sk}. In general, the surface $S_k$ is obtained from $S_{k-1}$ by adding a triangle to every edge of the boundary 
$\partial S_{k-1}$. It is clear that for the central triangle $\sigma_\ast$ of $S_k$, one has $D_{S_k}(\sigma_\ast)=k$. Thus $S_k$ is not collapsible to a graph in $k$ steps, but is collapsible in $k+1$ steps.}
\end{example}
\begin{figure}[h]
\begin{center}
\resizebox{10cm}{4.5cm}{\includegraphics[73,382][501,597]{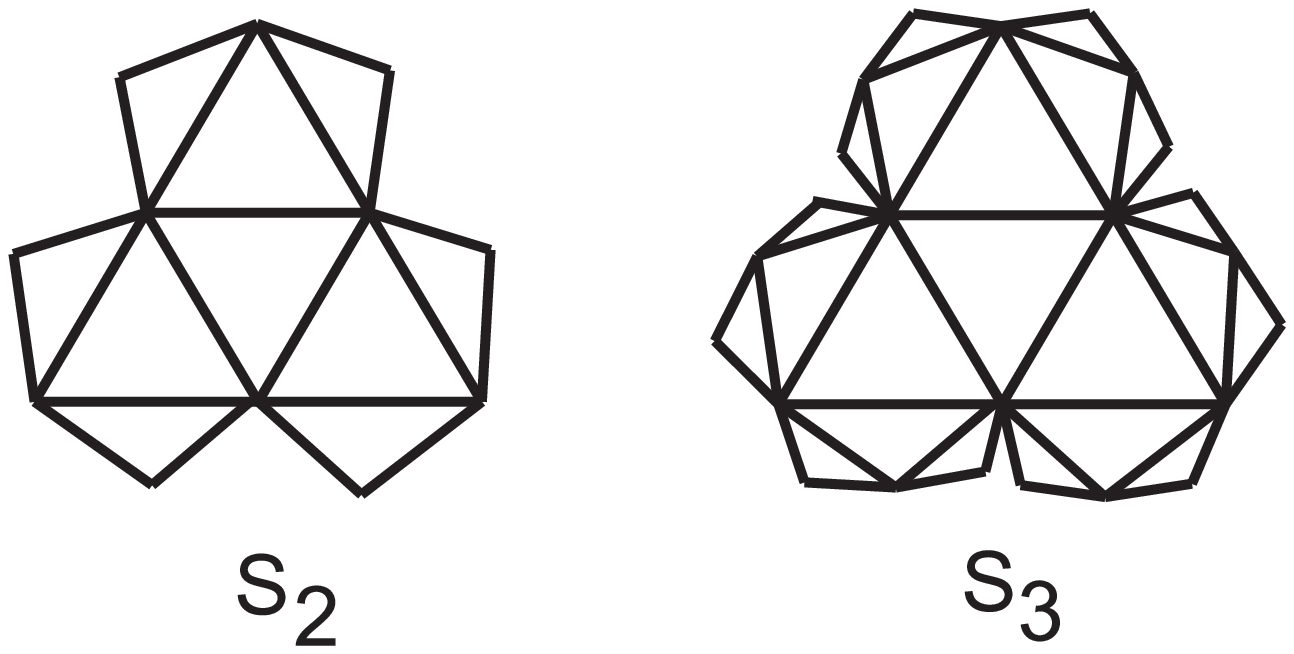}}
\end{center}
\caption{Surfaces $S_k\in \LL_{k,2}$.}\label{sk}
\end{figure}

The following Theorem plays a key role in this paper:

\begin{theorem}\label{thm1}
A  2-complex $Y$ of degree at most $r\ge 2$ is not collapsible to a graph  in $k$ steps, where $k=0, 1, 2, \dots$, if and only if there is a surface $S\in \LL_{k,r}$
which admits a simplicial embedding $S\to Y$. 
\end{theorem}

In the proof, we will use the following statement: 

\begin{lemma}\label{lm10}
Let $Y$ be a finite 2-dimensional simplicial complex of degree at most $r$ and let $\sigma$ be a 2-simplex in $Y$ with $D_Y(\sigma)=k$, where $k=0, 1, 2, \dots$.
Then there exists a surface $S\in \LL_{k,r}$ and a simplicial embedding $S\to Y$ such that the central simplex $\sigma_\ast$ of $S$ is mapped onto $\sigma$. 
\end{lemma}

\begin{proof}[Proof of Lemma \ref{lm10}] We will use induction on $k=D_Y(\sigma)$. 
For $k=0$, the statement is obvious. Assume that it is true for all cases with $D_Y(\sigma)<k$, and consider the situation when $D_Y(\sigma)=k>0$. If $Y\searrow Y'$ is the first collapse, then 
$\sigma\subset Y'$ and clearly%\marginpar{explain why} 
$$D_{Y'}(\sigma)=k-1$$
and $Y'$ has degree at most $r$. 
By the inductive hypothesis, there exists $S'\in \LL_{k-1,r}$ and a simplicial embedding $S'\to Y'$, mapping the central simplex of $S'$ onto $\sigma$. 

For each edge 
$e$ lying in $A_{S'}(\sigma)$ choose a 2-simplex $\sigma_e\subset Y$ as follows. If $e\subset \partial Y'$, let $\sigma_e$ be any free triangle  in $Y$ containing $e$. 
If $e\not\subset \partial Y'$, let $\sigma_e$ be any triangle in $Y'$ containing $e$ which is not in $S'$; such $\sigma_e$ exists since $e\not\subset \partial Y'$. 

Next we define a subcomplex $S\subset Y$ as the union
$$S=S'\cup \bigcup_{e}
\sigma_e \, \subset Y,$$
where $e$ runs over the edges in $A_{S'}(\sigma)$. Note that $S$ is finite, pure, and strongly connected since $S'$ is an $r$-pseudo-surface.  
Moreover, the degree of $S$ is at most $r$ since it is a subcomplex of $Y$. One has 
$D_S(\sigma)\ge k$
by Corollary \ref{corgood}. More precisely, we obtain that $D_S(\sigma)=k$
by Lemma \ref{ge}. Finally we observe that obviously $d_S(\sigma, \sigma')\le k$ for any 2-simplex $\sigma'$ of $S$.
Thus, $S\in \LL_{k,r}$. \end{proof}

\begin{proof}[Proof of Theorem \ref{thm1}] Consider the sequence of successive collapses $Y\searrow Y^{(1)}\searrow Y^{(2)} \searrow Y^{(3)}\searrow \dots$. 
We assume that $Y$ is not collapsible to a graph in $k$ steps, which implies that there are two possibilities: either (a) $Y^{(i)}\not= Y^{(i+1)}$ for any $i<k$;  or (b) for some $i<k$, 
one has $\partial Y^{(i)}=\emptyset$. 

In case (a), the complex $Y$ contains a 2-simplex with $D_Y(\sigma)=k$ and Lemma \ref{lm10} gives us an embedding of an $r$-pseudo-surface $S\in \LL_{k,r}$ into $Y$. 

In case (b), we have
$\partial Y^{(i)}=\emptyset$
for some $i<k$. Fix a 2-simplex $\sigma_\ast\in Y^{(i)}$ and consider distances $d_{Y^{(i)}}(\sigma_\ast, \sigma)$ to various 2-simplexes $\sigma$ of $Y^{(i)}$. If all these distances are less than or equal to $k$, then 
$Y^{(i)}$ belongs to $\LL_{k,r}$ and we are done. If there are simplexes $\sigma$ such that $d_{Y^{(i)}}(\sigma_\ast, \sigma)>k$, then consider the subcomplex 
$Z\subset Y^{(i)}$ defined as the union of all $\sigma$ with $d_{Y^{(i)}}(\sigma_\ast, \sigma)\le k$. 

Clearly $Z$ is not collapsible to a graph in $k$ steps. Therefore, in the sequence of collapses $Z\searrow Z^{(1)}\searrow Z^{(2)} \searrow Z^{(3)}\searrow \dots$, we again 
have either case (a) or (b) as above. In case (a), we apply Lemma \ref{lm10}; and in case (b), we obtain a subcomplex $S\subset Z$ with $\partial S=\emptyset$ such that
$d(\sigma_\ast, \sigma)\le k$ for any $\sigma\subset S$. We have $S\in \LL_{k,r}$ in either case, completing the proof. 
\end{proof}

\section{Collapsibility of a random 2-complex}

\subsection{The degree sequence}
Recall that the degree of an edge $e$ in a 2-complex is defined as the number of 2-simplexes which contain $e$. The degree of an edge in a random 2-complex 
$Y\in G(\Delta_n^{(2)},p)$ is an integer in the set $\{0, 1, \dots, n-2\}$.

%For a random 2-complex $Y$ we denote by 
%$$d_1\ge d_2\ge \dots \ge d_{n-2}$$
%the degree sequence. 

Let $X_k: G(\Delta_n^{(2)},p)\to \Z$ be the random variable counting the number of edges of degree $k$ in a random $2$-complex, where $k=0,1, 2, \dots, n-2$.
A straightforward calculation reveals that 
\[
\E(X_k) = {n \choose 2} {{n-2}\choose k}p^k(1-p)^{n-2-k}.\]
The expectation of the number of edges of degree at least $r$ in a random $2$-complex is 
\begin{equation} \label{edegree}
\sum_{k=r}^{n-2} \E(X_k) \le n^2\sum_{k=r}^{n-2}(pn)^k \le \frac{n^2(pn)^r}{1-pn}.
\end{equation}

\begin{corollary}\label{cor12} The probability that a random 2-complex $Y\in G(\Delta_n^{(2)}, p)$ has an edge of degree at least $r$ is less than or equal to  
$$\frac{n^{2+r}p^r}{1-pn}.$$ 
Thus, if $$p\ll n^{-1-\frac{2}{r}},$$
then a random 2-complex $Y\in G(\Delta_n^{(2)}, p)$ has no edges of degree $r$ or greater, a.a.s.
\end{corollary} 
\begin{proof} This follows from inequality \eqref{edegree} by applying the first moment method, see, for instance, \cite{JLR}. 
\end{proof}

\subsection{The invariant $\tilde \mu(S)$.} 
Following \cite{BHK} and \cite{CFK}, for a $2$-complex $S$ with $v=v(S)$ vertices and $f=f(S)>0$ faces one defines
\[
\mu(S) = \frac{v}{f} \in \Q,
\]
and 
\[
\tilde\mu(S)=\min_{S'\subset S} \mu(S'),
\]
where $S'$ runs over all subcomplexes of $S$ or, equivalently, over all pure subcomplexes $S'\subset S$. Note the following {\it monotonicity property} of $\tilde \mu$:
\begin{eqnarray}\label{mon}
\mbox{if} \quad S\subset T, \quad \mbox{then}\quad \tilde \mu(S) \ge \tilde \mu(T). 
\end{eqnarray}

The invariant $\tilde \mu$ controls embeddability of finite 2-complexes into random 2-complexes as illustrated by the following result.

\begin{theorem}[\cite{CFK}]\label{tilda} Let $S$ be a finite simplicial complex.  
\begin{enumerate}
  \item[(a)] If $p\ll n^{-\tilde \mu (S)}$, the probability that $S$ admits a  simplicial embedding into a random 2-complex $Y\subset G(\Delta_n^{(2)}, p)$ tends to zero as $n\to \infty$;
  \item[(b)] If $p\gg n^{-\tilde \mu (S)}$, the probability that $S$ admits a simplicial embedding into a random 2-complex $Y\subset G(\Delta_n^{(2)}, p)$ tends to one as $n\to \infty$. 
\end{enumerate}
\end{theorem}

\begin{definition} 
A 2-complex $S$ is called balanced if $\tilde \mu(S) = \mu(S)$, or, equivalently, $\mu(S')\ge \mu(S)$ for any subcomplex $S'\subset S$. 
\end{definition}

Any triangulated surface is balanced, see \cite{CFK}.

\begin{example}{\rm 
Suppose that a 2-complex $S$ has a free triangle with two free edges, and that the result $S'$ of removing this triangle satisfies $\mu(S')<1$. Then $\mu(S)>\mu(S')$ and $S$ is unbalanced. 
Indeed, if $\mu(S')=v/f$, where $v=v(S')$ and $f=f(S')$, then $v<f$ and we have $\mu(S) = (v+1)/(f+1)>v/f$. 
In this way one produces many unbalance 2-complexes, including 2-disks. }
\end{example}

Next, we examine the $\tilde \mu$ invariants of 2-complexes $S\in \LL_{k, r}$.

\begin{lemma}\label{closed} Let $S$ be a closed 2-complex, i.e., $\partial S=\emptyset$. Then $\tilde \mu(S)\le 1$. 
\end{lemma}
\begin{proof} Without loss of generality, we may assume that $S$ is connected, since otherwise we can apply the following arguments to a connected component of $S$ and use the monotonicity property (\ref{mon}). Moreover, we may assume that $S$ is pure, since otherwise we may deal with the maximal pure subcomplex of $S$ instead of $S$. 

Suppose first that $H_2(S;\Z_2)=0$. Then by the Euler--Poincar\'e theorem, $\chi(S)\le 1$, and we have
$$v-e+f=\chi(S)\le 1,\quad \mbox{and} \quad 3f\ge 2e,$$
where $v, e, f$ denote the numbers of vertices, edges and faces in $S$. In the latter inequality we used the assumptions that $S$ is pure and closed. These inequalities imply
$$v-f/2 \le \chi(S) \le 1, \quad \mbox{and}\quad \mu(S) \le 1/2 + 1/f.$$
Since $f\ge 4$ we obtain that $\tilde \mu(S) \le \mu(S)\le 3/4 <1.$

Assume now that $H_2(S;\Z_2)\not=0$. We will show that there is a subcomplex $S'\subset S$ which is also closed, $\partial S'=\emptyset$, and satisfies $H_2(S';\Z_2)=\Z_2$. 
Indeed, consider a nonzero two-dimensional cycle $c=\sum_{i\in I} \sigma_i$ with $\Z_2$ coefficients, where the $\sigma_i$ are distinct 2-simplexes of $S$. Let $I' \subseteq I$ be the minimal subset of the indexing set $I$ for which $c'=\sum_{i\in I'} \sigma_i$ is still a cycle, and let $S'=\bigcup_{i\in I'} \sigma_i$ be the corresponding subcomplex of $S$.  Then clearly $H_2(S';\Z_2)=\Z_2$ and $S'$ is closed and pure. 

By the Euler--Poincar\'e theorem, $\chi(S')\le 2$, and we have 
$$v'-e'+f'=\chi(S') \le 2, \quad \mbox{and}\quad 3f'\ge 2e',$$
where $v',e ', f'$ denote the numbers of vertices, edges and faces in $S'$. This gives
$$v'-f'/2\le \chi(S') \le 2,$$
and
\begin{eqnarray}\label{two}
\mu(S') \le \frac{1}{2} + \frac{2}{f'}. 
\end{eqnarray}
Since $f'\ge 4$, the last inequality gives $\mu(S')\le 1$. Finally, we have $\tilde \mu(S) \le \mu(S') \le 1$. 
\end{proof}
%
%Note that our arguments prove a slightly stronger statement: {\it If $\partial S=\emptyset$ then $\tilde \mu(S) \le 9/10$ unless $S$ contains a tetrahedron as a subcomplex}. Indeed, 
%(\ref{two}) implies that $\mu(S') \le 9/10$ and therefore $\tilde \mu(S) <9/10$ unless $f'=4$, that is $S'$ is a tetrahedron.

\begin{lemma}\label{nonclosed} If $S\in \LL_{k,r}$ for some $k\ge 0$, \, $r\ge 2$ then one has 
\begin{eqnarray}\label{finally}
\tilde \mu(S) \le 1+\frac{2}{k+1}.\end{eqnarray}
\end{lemma}
\begin{proof} If $S$ is closed the result follows from Lemma \ref{closed}. 
Assume now that $\partial S\not=\emptyset$. 
Let $\sigma_\ast$ be the central simplex of $S$ and let $\sigma_0, \sigma_1, \dots, \sigma_k=\sigma_\ast$ be a collapsing path leading to $\sigma_\ast$. 
Here $D_S(\sigma_i) = i$ and $\sigma_i\cap \sigma_{i+1}$ is an edge, see Definition \ref{defaccess}. 
Then the union 
$S'=\cup_{i=0}^k\sigma_i$ is a subcomplex having exactly $k+1$ faces and at most $k+3$ vertices.
Thus, 
$$\mu(S') \le \frac{k+3}{k+1}=1+\frac{2}{k+1},$$ 
establishing (\ref{finally}). 
\end{proof}

\subsection{The threshold for $k$-collapsibility.}
\begin{definition}
Let $\tilde\mu_{k,r}$ denote the largest possible value of the invariant $\tilde\mu(S)$ for $S$ a forbidden $r$-pseudo-surface, 
\[
\tilde \mu_{k,r} \, = \, \max_{S\in \LL_{k,r}}\tilde \mu(S) \, \in \Q.
\]
\end{definition}

For instance, examining the surfaces shown in Figure \ref{lone} reveals that $\tilde \mu_{1,2} =3/2$.

\begin{theorem}\label{thm21} Consider a random 2-complex $Y\in G(\Delta_n^{(2)},p)$. 
\begin{enumerate}
\item[(a)] If for some $r\ge 2$ and $k\ge 1$, one has $$p\ll n^{-1-\frac{2}{r+1}} \quad \mbox{and}\quad p\ll n^{-\tilde \mu_{k,r}},$$ then $Y$ is collapsible to a graph in at most $k$ steps, a.a.s.
\item[(b)] If for some $r\ge 2$ and $k\ge 1$, one has $p\gg n^{-\tilde \mu_{k,r}}$, then $Y$ is not collapsible to a graph in $k$ or fewer steps, a.a.s.
\end{enumerate}
\end{theorem}
\begin{proof} By Corollary \ref{cor12}, if $p\ll n^{-1-\frac{2}{r+1}}$, then a random 2-complex $Y\in G(\Delta_n^{(2)},p)$ has degree at most $r$, a.a.s. 
Next, we apply Theorem \ref{thm1} and examine the embeddability of complexes $S\in \LL_{k,r}$ into $Y$. 
By Theorem \ref{tilda} (a), if $p\ll n^{-\tilde \mu(S)}$, then 
$S$ does not embed 
into $Y$, a.a.s. Since $\tilde \mu_{k,r}\ge \tilde \mu(S)$, we see that the assumption $p\ll n^{-\tilde \mu_{k,r}}$ implies that no $S\in \LL_{k,r}$ can be embedded into $Y$, a.a.s.
Thus, by Theorem \ref{thm1}, we see that $Y$ is collapsible to a graph in $ k$ or fewer steps. This proves part (a). 

To prove part (b), we apply Theorem \ref{tilda} (b) to conclude that if $p\gg n^{-\tilde \mu_{k,r}}$, then there exists $S\in \LL_{k,r}$ which is embeddable into $Y$, a.a.s. 
This implies that 
$Y$ is not collapsible to a graph in at most $k$ steps, a.a.s.
\end{proof}

\begin{example} {\rm  Consider the surface $S_k\in \LL_{k,2}$ introduced in Example \ref{ex1}. Note that $S_k\in \LL_{k,r}$ for any $r\ge 2$. The numbers of vertices $v_k$ and faces $f_k$ of $S_k$ satisfy
the recurrence relations
\begin{eqnarray}\label{musk}v_k=2\cdot v_{k-1}\quad \mbox{and}\quad f_k = v_{k-1}+f_{k-1}.\end{eqnarray}
Indeed, viewing $S_{k-1}$ as a subcomplex of $S_k$, we see that all vertices of $S_{k-1}$ lie on the boundary, and each edge of the boundary of $S_{k-1}$ adds a vertex to $S_k$.  This explains the first equation. For the second, note that the number of new triangles
in $S_k$ is equal to the number of edges on $\partial S_{k-1}$. 

Since $v_0=3$ and $f_0=1$, solving the recurrence relations \eqref{musk} yields 
$$v_k = 3\cdot 2^{k} \quad \mbox{and}\quad f_k = 3\cdot 2^{k}-2. $$
Consequently, 
$$\mu(S_k)= 1+ \frac{1}{3\cdot 2^{k-1} -1}.$$
}
\end{example}

\begin{lemma}\label{star}
The surface $S_k$ is balanced, and hence $$\tilde \mu(S_k) = \mu(S_k)= 1+ \frac{1}{3\cdot 2^{k-1} -1}.$$
\end{lemma} 
\begin{proof} 
Let $S$ be a pure subcomplex of $S_k$ with $v=v(S)$ vertices and $f=f(S)$ faces.  Write $v=v_k-m$ and $f=f_k-n$, where 
$v_k$ and $f_k$ are as above and
$m$ and $n$ are the number of vertices and faces which are in $S_k$, but not in $S$.  We claim that $m=v_k-v \le f_k-f =n$.  This assertion is established by induction.

The case $k=0$ is trivial.  So assume inductively that for any $i<k$ and $S'
\subset S_i$  a pure subcomplex, we have $v(S_i)-v(S') \le f(S_i)-f(S')$.

For a pure subcomplex $S \subset S_k$ as above, let $S'$ be the pure part of $S\cap S_{k-1}$. 
Then, $m=m'+m''$ and $n=n'+n''$, where $v(S')=v_{k-1}-m'$, $f(S')=f_{k-1}-n'$, $m''$ is the number of vertices in $S_k\smallsetminus S_{k-1}$ which are not in $S$, and $n''$ is the number of faces in $S_k\smallsetminus S_{k-1}$ which are not in $S$.

We have $m'\le n'$ by induction.  Observe that the vertices of $S_k\smallsetminus S_{k-1}$ are in one-to-one correspondence with the faces of $S_k\smallsetminus S_{k-1}$.  If such a vertex is not in $S$, then the corresponding face cannot be in $S$ either.  Consequently, $m''=n''$, and $m=m'+m''\le n'+n''=n$, completing the proof of the claim.

It follows immediately that $\mu(S)\ge\mu(S_k)=\mu_k$.  Indeed, \[ \frac{v}{f}-\frac{v_k}{f_k}=\frac{v_k-m}{f_k-n}-\frac{v_k}{f_k}=\frac{nv_k-mf_k}{f_k(f_k-n)}=
\frac{\mu_k
n - m}{f_k-n} \ge \frac{n-m}{f_k-n}\ge 0.
\]
Thus, $S_k$ is balanced.
\end{proof}

{F}rom Lemmas %\ref{closed}, 
\ref{nonclosed} and \ref{star} we obtain:

\begin{corollary}\label{cor23} For any $r\ge 2$ and $k\ge 0$, one has the following inequalities:
\[
1+ \frac{1}{3\cdot 2^{k-1} -1}\, \le \, \tilde \mu_{k,r} \, \le\,  1+ \frac{2}{k+1}.
\]
\end{corollary}
Note that the obtained upper and lower bounds for $\tilde \mu_{k,r}$ are independent of $r$. 

We believe that $\tilde\mu_{k,r}=1+  1/(3\cdot 2^{k-1} -1)$.

%We expect that the upper bound $\tilde \mu_{k,r} \, \le 1+ \frac{2}{k+1}$ can be improved.  
%Moreover, we believe that $\tilde\mu_{k,r}=1+ \frac{1}{3\cdot 2^{k-1} -1}$.

\begin{proof}[Proof of Theorem \ref{main}] 
The main theorem is now an immediate consequence of Theorem \ref{thm21} and Corollary \ref{cor23}:

(a) Assume that $p \ll n^{-1-2/(k+1)}$ for some $k \ge 1$. According to Corollary \ref{cor23},  
$\tilde \mu_{k,r} \le 1 + 2/(k+1).$
Choosing $r= \mbox{max} (2,k)$, it then follows from Theorem \ref{thm21} (a)    
that $Y\in G(\Delta_n^{(2)}, p)$ is collapsible to a graph in at most $k$ steps, a.a.s. 

(b) Assume that $p \gg  n^{-1-1/(3\cdot 2^{k-1} - 1)}$ for some $k \ge 1$. Then 
by Theorem \ref{tilda} and Lemma \ref{star} the surface $S_k$ (see Example \ref{ex1}) embeds into $Y$, a.a.s.
Since $S_k$ cannot be collapsed to a graph in $k$ or fewer steps we obtain that $Y$ is not collapsible to a graph in $k$ or fewer steps. 
\end{proof}

\newcommand{\arxiv}[1]{{\texttt{\href{http://arxiv.org/abs/#1}{{arXiv:#1}}}}}

\newcommand{\MRh}[1]{\href{http://www.ams.org/mathscinet-getitem?mr=#1}{MR#1}}

\vskip 2cm

Daniel C. Cohen

Department of Mathematics

Louisiana State University

Baton Rouge, LA 70803 USA

cohen@math.lsu.edu

www.math.lsu.edu/$\sim$cohen
\vskip 1cm

Michael Farber

Department of Mathematical Sciences

Durham University

Durham, DH1 3LE, UK

Michael.farber@durham.ac.uk

http://maths.dur.ac.uk/$\sim$dma0mf/

\vskip 1cm

Thomas Kappeler

Mathematical Insitutte 

University of Zurich

Winterthurerstrasse 190, CH-8057 

Zurich, Switzerland

thomas.kappeler@math.uzh.ch

   % Enter section title between curly braces

%\subsection{}    % Enter subsection title between curly braces

% Set the ending of a LaTeX document
\end{document}